\documentclass[12pt]{article}
\usepackage{a4}
\usepackage{amsthm}
\usepackage{amsfonts}
\usepackage{amsmath}
\newtheorem{theorem}{Theorem}
\newtheorem{lemma}[theorem]{Lemma}
\newtheorem{proposition}[theorem]{Proposition}
\newtheorem{corollary}[theorem]{Corollary}
\def\CC{{\cal C}}
\def\FF{{\cal F}}
\def\VF{{\rm VF}}
\def\OO{{\cal O}}
\def\VFodd{{\rm VF}_{\rm odd}}
\def\dead{{\rm dead}}
\def\edge{_e}
\def\triangle{_t}
\begin{document}
\title{Min-max relations for odd cycles in planar graphs\thanks{This research was done in the framework of the Czech-French project PHC Barrande 24444XD (the Czech side reference: MEB021115).}}
\author{Daniel Kr\'al'\thanks{Department of Applied Mathematics and Institute for Theoretical Computer Science (ITI), Faculty of Mathematics and Physics, Charles University, Malostransk\'e n\'am\v est\'\i{} 25, 118 00 Prague 1, Czech Republic. E-mail: \texttt{kral@kam.mff.cuni.cz}. Institute for Theoretical computer science is supported as project 1M0545 by Czech Ministry of Education.}\and
        Jean-S{\'e}bastien Sereni\thanks{CNRS (LIAFA, Universit\'e Denis Diderot), Paris, France, and Department of Applied Mathematics (KAM), Faculty of Mathematics and Physics, Charles University, Prague, Czech Republic.  E-mail: \texttt{sereni@kam.mff.cuni.cz}.}\and
	Ladislav Stacho\thanks{Department of Mathematics, Simon Fraser University, 8888 University Drive, Burnaby, B.C., V5A 1S6, Canada. E-mail: {\tt lstacho@sfu.ca}. This author's stay at LIAFA was partially supported by the French {\em Agence nationale de la recherche} under the reference ANR 10 JCJC 0204 01.}
	}
\date{}
\maketitle
\begin{abstract}
Let $\nu(G)$ be the maximum number of vertex-disjoint odd cycles of a graph $G$ and
$\tau(G)$ the minimum number of vertices whose removal makes $G$ bipartite.
We show that $\tau(G)\le 6\nu(G)$ if $G$ is planar. This improves the previous
bound $\tau(G)\le 10\nu(G)$ by Fiorini, Hardy, Reed and Vetta [Math.~Program.~Ser.~B 110 (2007), 71--91].
\end{abstract}

\section{Introduction}
\label{sec-intro}

Packing problems are among the most important problems in combinatorial optimization.
In this paper, we focus on the problem of packing odd cycles in graphs.
If $G$ is a graph,
let $\nu(G)$ be the size of a maximum collection (packing) of vertex-disjoint odd cycles of $G$, and
$\tau(G)$ the size of a minimum set $S$ of vertices (transversal) such that each odd cycle of $G$
contains at least one vertex of $S$ (which is is equivalent to $G\setminus S$ being bipartite).
Clearly, $\nu(G)\le\tau(G)$.

One of the most studied questions on packing problems is whether the size of a maximum packing can be bounded
by a function of the size of a minimum transversal. If this is the case, the problem is said
to have the {\em Erd{\H o}s-P\'osa property}. The name is due to the result of Erd{\H o}s and
P\'osa~\cite{bib-erdos65+} who established this property for packing (general) cycles
in graphs.
For general graphs, the packing problem for odd cycles does not have the Erd{\H o}s-P\'osa property.
Reed~\cite{bib-reed99} gave an example of a projective planar graph $G$ with $\tau(G)$
arbitrary large and no two vertex-disjoint odd cycles, i.e., $\nu(G)=1$.
These graphs are called Escher walls. In fact, they play a key role for the problem.
The main result from~\cite{bib-reed99} asserts that the problem of packing odd cycles
in a minor-closed family of graphs has the Erd{\H o}s-P\'osa property if and only if
the family avoids Escher walls of arbitrary height.

Since the class of planar graphs avoids all Escher walls, it follows that there exists
a function $f$ such that $\tau(G)\le f(\nu(G))$ if $G$ is planar. However, the function
given by the methods from~\cite{bib-reed99} is enormous since the proof is based
on the graph minor machinery. So, it is natural to search for better estimates
for particular graph classes.
In~\cite{bib-fiorini07+}, Fiorini, Hardy, Reed and Vetta showed that
$\tau(G)\le 10\nu(G)$ for planar graphs.
The purpose of this article is to further improve this estimate
to $\tau(G)\le 6\nu(G)$ (Theorem~\ref{thm-main}).
These results also hold in a more general setting where edges are assigned parities.
Since our proof is constructive and all its steps can be efficiently performed,
we also obtain the existence of a polynomial-time $6$-approximation algorithm
for the odd cycle packing problem in planar graphs (Corollary~\ref{cor-alg}) which
improves the $11$-approximation algorithm given in~\cite{bib-fiorini07+}.
The problem is known~\cite{bib-hardy05} to be NP-hard.

We do not believe the multiplicative constant in Theorem~\ref{thm-main} is optimal.
In fact, we are not aware of an example showing it exceeds two.
This multiplicative factor is known~\cite{bib-rautenbach01+} to be true,
i.e., $\tau(G)\le 2\nu(G)$,
if $G$ is highly connected (the connectivity depends on $\nu(G)$).
The optimal constant is however known for the edge version of the problem in the planar case.
Similarly to the vertex case, there is no function bounding the edge transversal $\tau\edge(G)$
in terms of the size $\nu\edge(G)$ of a maximum collection of edge-disjoint odd cycles for general graphs $G$.
However, for planar graphs, such a function exists~\cite{bib-berge00+}, and
the optimal estimate $\tau\edge(G)\le 2\nu\edge(G)$ was proven in~\cite{bib-kral04+};
its more compact proof can be found in~\cite{bib-fiorini07+}.

Another related problem is a conjecture of Tuza which asserts that
the minimum size $\tau\triangle(G)$ of a set of edges intersecting every triangle
is at most twice the maximum number $\nu\triangle(G)$ of edge-disjoint triangles.
The conjecture is known to be true for planar graphs~\cite{bib-tuza90} and
it is also known that two of its fractional relaxations hold~\cite{bib-krivelevich95}.

\section{Notation}

We work in the more general setting of {\em signed} graphs. In this setting,
each edge is assigned a parity, i.e., it is odd or even. A cycle is said
to be {\em odd} if it contains an odd number of odd edges and it is {\em even}
otherwise. A face of a plane signed graph is {\em odd} if its boundary contains
an odd number of odd edges (bridges are counted twice); it is {\em even} otherwise.
It is easy to see that the boundary of an odd face must contain an odd cycle.
The property whether a cycle is odd or even is referred to as its {\em parity}.

Since we exclusively focus on odd cycles in signed graphs,
we call a set $S$ of vertices of a signed graph $G$ a {\em transversal} of $G$
if $G\setminus S$ has no odd cycle. A collection $\CC$ of cycles is a {\em packing}
if the cycles in $\CC$ are vertex-disjoint. The two parameters central to our study
are $\tau(G)$ which stands for the minimum size of a transversal of $G$ and
$\nu(G)$ which is the maximum number of odd cycles forming a packing in $G$.
In case that all edges of a signed graph $G$ are odd, a cycle in $G$ is odd
if and only if its length is odd, so the definitions coincide with those
given in Section~\ref{sec-intro}.

Signed graphs we consider in this paper are always assumed to be simple.
This does not decrease the generality of our results: if a signed graph $G$
contains parallel edges, we can subdivide each parallel edge in such a way that
one of the new edges has the same parity as the original edge and the other one is even.
It is not hard to observe that this operation affects neither $\tau(G)$ nor $\nu(G)$ (considering
pairs of parallel edges with different parities as odd cycles of length two).

\subsection{$T$-joins and $T$-cuts}

A key ingredient for our proof is the notion of a $T$-join from combinatorial optimization.
The algorithm for finding a minimum-size $T$-join forms the core of the algorithm for solving
the max-cut problem for planar graphs. The planar max-cut problem is dual to finding the minimum set of edges
whose removal bipartizes a given planar graph, which is the quantity $\tau\edge$ defined earlier.
So, it is not surprising that the proof of an earlier bound in~\cite{bib-fiorini07+} as well as
our proof use this notion. In fact, the arguments we use in Section~\ref{sec-tjoin}
can be viewed as an extension of those given in~\cite[Subsection 4.3]{bib-fiorini07+}.

We now present the notion and results we later need. The reader can find a more detailed
introduction in monographs on combinatorial optimization, e.g., \cite{bib-cook98+,bib-schrijver03}.
A {\em $T$-join} in a connected graph $G$ with a distinguished even-size set $T$ of its vertices
is a subgraph $J$ such that the odd degree vertices of $J$ are precisely the vertices of $T$.
The {\em size} of a $T$-join $J$ is the number of edges it contains and it is denoted by $|J|$.
The problem of finding a minimum-size $T$-join
can be reduced to the weighted perfect matching problem on complete graphs
which is well-understood and efficiently solvable. The reduction is as follows:
form a complete graph with vertex set $T$ and assign each edge $tt'$ the length $d_G(t,t')$ of the shortest
path between $t$ and $t'$ in $G$. For a minimum weight perfect matching in the auxiliary graph,
define a subgraph $J$ to be the union of the shortest paths corresponding to the edges of the matching (it can
be shown that the paths are edge-disjoint if the perfect matching has minimum weight). So, $J$ forms
a $T$-join in $G$ which is minimum.

It is well-known that the problem of finding a minimum weight perfect matching
can be formulated as a linear program.
Considering its dual, we obtain the following optimization problem
with variables $y_v$, $v\in T$, and $y_S$, $S\in\OO(T)$,
where $\OO(T)$ stands for the set of all odd-size subsets of $T$ with at least three elements.
$$
\begin{array}{rcll}
y_S & \ge & 0 & \mbox{for $S\in\OO(T)$} \\
y_t+y_{t'}+\sum_{S\in\OO(T), |S\cap\{t,t'\}|=1} y_S & \le & d_G(t,t') & \mbox {for every pair $t,t'\in T$} \\
\\
\min \sum_{t\in T}y_t+\sum_{S\in\OO(T)}y_S
\end{array}
\qquad \mbox{(Y)}
$$
The duality of linear programming implies that the optimum value of the linear program (Y) is equal to
the minimum size of a $T$-join in a graph $G$.
A solution of (Y) is called {\em laminar} if $y_S>0$ and $y_{S'}>0$ implies that
either $S$ and $S'$ are disjoint or one is a subset of the other.
It is well-known that the linear program (Y) has always an optimum laminar solution.
Moreover, since the weights of the edges satisfy the triangle inequality,
there is an optimum laminar solution of (Y) with all variables being non-negative.
In case that the weights of the edges in the auxiliary graph are even (which happens, e.g., when $G$ is bipartite),
there always exists an optimum integral solution of (Y) which is also non-negative and laminar.
We summarize these observations in the next proposition.

\begin{proposition}
\label{prop-dual}
Let $G$ be a connected graph with a distinguished even-size set $T$ of its vertices.
The value of the optimum solution of the linear program (Y) is equal to the minimum size of a $T$-join of $G$.
Moreover, there exists an optimum solution of (Y) that is non-negative and laminar, and
if $G$ is bipartite, there exists an optimum solution that is non-negative, laminar and integral.
\end{proposition}

By the duality of linear programming, if a $T$-join $J$ in a graph $G$ has the size equal to a value of
a solution $y$ of (Y), then $J$ is a minimum size $T$-join and $y$ is an optimum solution of (Y).
We use this fact in the proof of Lemma~\ref{lm-deadly} where we keep such a pair through the induction,
so the $T$-joins we consider are optimum.

A combinatorial structure dual to a $T$-join is a $T$-cut: a {\em $T$-cut} is an edge cut that
splits the graph into two parts each containing an odd number of vertices of $T$. It is known~\cite{bib-seymour81} that
if $G$ is bipartite, then the size of a minimum $T$-join is equal to the maximum number of edge-disjoint
$T$-cuts of $G$.
More insight in the structure of optimum $T$-joins and solutions of (Y)
can be derived by analyzing specific procedures for obtaining them.
We state one such condition based on the blossom algorithm for the (weighted) perfect matching problem.
To do so, we define a vertex $v$ of a graph $G$ to be {\em $k$-close} to a set $S\in\OO(T)$
with respect to a solution $y$ of (Y) if
$$\min_{t\in S}\left(d_G(t,v)-y_t-\sum_{\substack{S'\in\OO(S)\\t\in S'\not=S}}y_{S'}\right)=k$$
where $d_G(t,v)$ is the distance between $t$ and $v$ in $G$.
A solution of (Y) is a {\em moat solution} if 
\begin{itemize}
\item $y$ is non-negative, integral and laminar,
\item for every inclusion-wise minimal set $S\in\OO(T)$ with $y_S>0$, there exists an ordering
      $t_1,\ldots,t_s$ of vertices of $S$ such that $y_{t_i}+y_{t_{i+1}}=d_G(t_i,t_{i+1})$
      for every $i=1,\ldots,s$ (indices modulo $s$), and
\item $\bigcup_{t\in T}\CC_t\cup\bigcup_{S\in\OO(T)}\CC_S$ is a collection of edge-disjoint $T$-cuts
      where $\CC_t$, $t\in T$, is the collection of $y_t$ $T$-cuts formed by the edges joining
      pairs of vertices at distance $k$ and $k+1$ from $t$ in $G$ for $k=0,\ldots,y_t-1$, and
      $\CC_S$, $S\in\OO(T)$, is the collection of $y_S$ $T$-cuts formed by the edges joining
      pairs of vertices that are $k$-close and $(k+1)$-close to $S$ in $G$ for $k=0,\dots,y_{S}-1$.
\end{itemize}
A moat solution always exists if $G$ is bipartite. Let us state this as a separate proposition.
\begin{proposition}
\label{prop-moats}
Let $G$ be a connected bipartite graph with a distinguished even-size set $T$ of its vertices.
There exists an optimum solution of (Y) that is a moat solution.
\end{proposition}
Observe that every optimum $T$-join intersects every $T$-cut in the collection $\bigcup_{t\in T}\CC_t\cup\bigcup_{S\in\OO(T)}\CC_S$
from the definition of a moat solution and
this intersection is formed by a single edge: this follows from that every $T$-join intersects every $T$-cut and
the size of an optimum $T$-join is equal to the number of $T$-cuts in the collection.

\section{Faces in clouds}

In this section, we analyze sets of odd faces that are ``connected'' in a considered plane signed graph.
Formally, we define the vertex-face incidence graph $\VF(G)$ of a plane signed graph $G$
to be the bipartite graph with vertex set formed by vertices and faces of $G$ such that
the vertex of $\VF(G)$ associated with a face $f$ of $G$ is adjacent to the vertices of $G$ incident with $f$.
The subgraph of $\VF(G)$ induced by the vertices of $G$ and the odd faces of $G$ is denoted by $\VFodd(G)$.
Finally, a {\em cloud} is a set of all odd faces in the same component of the graph $\VFodd(G)$.

We start with the following lemma.

\begin{lemma}
\label{lm-cloud-aux}
Let $G$ be a plane $3$-connected signed graph and $R$ a cloud of $G$.
There exists a set $\FF$ of vertex-disjoint faces of $R$ and
a set of vertices $U$ with the following properties:
\begin{description}
\item[P1] each face of $R$ is incident with at least one vertex of $U$,
\item[P2] each vertex of $U$ is incident with a face of $\FF$, and
\item[P3] $|U|\le 5|\FF|-1$.
\end{description}
\end{lemma}

\begin{proof}
The construction is iterative: at the beginning, we set $\FF_0$ and $U_0$
to be empty sets, and at each step we enlarge $\FF_k$ to $\FF_{k+1}$ by adding
a single face $f$ and $U_k$ to $U_{k+1}$ by adding at most five vertices
incident with $f$. So, $|\FF_k|=k$.
The sets $\FF_k$ and $U_k$ will satisfy the following:
\begin{description}
\item[Q1] the faces of $\FF_k$ are vertex-disjoint, and
\item[Q2] if a face $f'$ shares a vertex with a face $f\in\FF_k$, then $f'$ is incident to a vertex of $U_k$.
\end{description}

Suppose that $\FF_k$ and $U_k$ have already been constructed.
If every face of $R$ is incident to a vertex of $U_k$, set $\FF=\FF_k$ and $U=U_k$ (we verify
at the end of the proof the properties P1, P2 and P3).
So, we assume that there is a face of $R$ incident to no vertex of $U_k$.
To make our presentation clearer, let us call faces of $R$ incident
to vertices of $U_k$ {\em blocked}; the remaining faces of $R$ are
called {\em free}.

We now construct an auxiliary plane graph $H$ in the following way.
The vertices of $H$ are free faces. For every vertex $u$ of $G$ incident
to at least three free faces, say $f_1,\ldots,f_d$ (the faces are listed
in the cyclic order around $u$), add edges between $f_i$ and $f_{i+1}$,
$i=1,\ldots,d$ (indices modulo $d$). We call these edges {\em $u$-edges}.
If two free faces are incident but they do not share a vertex contained
in at least three free faces, we pick an arbitrary vertex $u$ they share and
add an edge to $H$ between the two faces and we call this edge a {\em $u$-edge}.
The construction of $H$ implies that $H$ is a planar graph since there is an embedding of $H$
naturally inherited from $G$; so we consider $H$ as a plane graph.
Since $G$ is $3$-connected, any two faces of $G$ share at most two vertices and
therefore there are at most two parallel edges between any two vertices of $H$.
Moreover, there are two edges between two vertices of $H$ (if and) only if
the associated faces of $G$ share two vertices and
both these vertices are contained in at least three free faces.

Let $H'$ be the plane graph obtained from $H$ by removing from every pair of parallel edges
one of the edges. Since $H'$ is plane, it contains a vertex of degree at most five.
Let $f$ be the face of $R$ associated with a vertex of minimum degree in $H'$. Let $u_1,\ldots,u_\ell$
be the vertices of $G$ such that $f$ is incident with a $u_i$-edge for $i=1,\ldots,\ell$.
Let $d$ be the degree of $f$ in $H$, $d'$ its degree in $H'$, and
$d_3$ the number of indices $i$ such that $u_i$ is shared by at least three free faces of $G$.
Since the number of pairs of parallel edges incident with $f$ is at most $d_3$,
it holds that $d-d'\le d_3$. On the other hand, it holds that $\ell=d-d_3$ ($f$ is incident
with two $u_i$-edges if and only if $u_i$ is shared by three free faces).
So, we obtain that $\ell=d-d_3\le d'\le 5$.

Now set $F_{k+1}=F_k\cup\{f\}$ and $U_{k+1}=U_k\cup\{u_1,\ldots,u_\ell\}$.
We verify that $F_{k+1}$ and $U_{k+1}$ satisfy Q1 and Q2.
The faces of $F_{k+1}$ are disjoint: if $f$ shared a vertex with a face of $F_k$,
it would be incident with a vertex of $U_k$ by Q2 and $f$ would be blocked.
To verify Q2, consider a face $f'$ that shares a vertex with a face of $F_{k+1}$.
If $f'$ is blocked,
it is incident to a vertex of $U_k$. If $f'$ is free (which includes the case $f'=f$),
it must be incident to at least one of the vertices $u_1,\ldots,u_\ell$.

To finish the proof, it remains to argue that the resulting sets $\FF$ and $U$
satisfy the properties P1, P2 and P3. The property P1 is satisfied
since we have stopped when all faces of $R$ are blocked. The property P2
follows from the fact that at each step we added to $U$ at most five vertices,
all of them incident with a face added to $\FF$ at that step.
The construction implies that $|U|\le 5|\FF|$. To verify P3, which asserts that
$|U|\le 5|\FF|-1$, we show that at most four vertices are added to $U$ in the last step.

Consider the graphs $H$ and $H'$ from the last step of the construction;
let $f$ be the face of $R$ added to $\FF$ at this step. If the degree $d'$ of $f$ in $H'$
is at most four, at most $d'\le 4$ vertices are added to $U$. So, we can assume that $d'=5$.
Consequently, the minimum degree of $H'$ is five by the choice of $f$.
Next, let $w_1,\ldots,w_5$ be the neighbors
of $f$ in $H'$ in a cyclic order around $f$ and $f_i$, $i=1,\ldots,5$,
the face of $H'$ containing the vertex associated with $f$, the vertex $w_i$ and the vertex $w_{i+1}$ (indices
modulo 5). Note that some of the faces $f_1,\ldots,f_5$ can coincide.

Let $u_1,\ldots,u_5$ be the vertices added to $U$ at the last step. Since the construction of $\FF$ and $U$
terminates after this step, every free face $f'$ contains one one of the vertices $u_1,\ldots,u_5$.
Hence, the vertex of $H'$ associated with $f'$ is incident with (at least) one of the faces $f_1,\ldots,f_5$ in $H'$.
We now derive a new plane graph $H''$ from $H'$. Insert in each face $f_i$ a new vertex $w'_i$, $i=1,\ldots,5$,
join $w'_i$ to the vertex associated with $f$, the vertex $w_i$ and the vertex $w_{i+1}$ (again, indices
modulo 5). In addition, join each vertex of $H'$ different from the vertex associated with $f$ and
the vertices $w_1,\ldots,w_5$
to one of the vertices $w'_1,\ldots,w'_5$ in such a way that the resulting graph is plane.

We now count the degrees of vertices of $H''$. Let $n$ be the number of vertices of $H''$.
Each vertex $v$ of $H'$ associated with a free face different from $f$ has degree at least six in $H''$
since the minimum degree of $H'$ is five and
each such vertex $v$ is joined to at least one of the vertices $w'_1,\ldots,w'_5$,
The degree of the vertex associated with $f$ is $10$ and the degrees of the vertices
$w'_1,\ldots,w'_5$ are at least three. So, the sum of the degrees of vertices of $H''$ is at least $6(n-6)+10+5\cdot 3=6n-11$.
However, the sum of the degrees cannot exceed $6n-12$ since $H''$ is plane. This contradicts our assumption that
$d'=5$ in the last step of the construction.
\end{proof}

We are now ready to prove the main lemma of this section.

\begin{lemma}
\label{lm-cloud}
Let $G$ be a plane $3$-connected signed graph and $R$ a cloud of $G$.
There exists a set $W$ of vertices of $G$ such that the subgraph of $\VFodd(G)$
induced by $R\cup W$ is connected and $|W|\le 6\nu(R)-2$,
where $\nu(R)$ is the maximum number of vertex-disjoint faces of $R$.
In particular, every face of $R$ is incident with a vertex of $W$.
\end{lemma}

\begin{proof}
Let $\FF$ and $U$ be the set of vertex-disjoint odd faces and vertices satisfying
properties P1, P2 and P3 from Lemma~\ref{lm-cloud-aux}. Set $W=U$.
Consider the component of
$\VFodd(G)$ corresponding to $R$ and the subgraph $H$ of it induced by $R\cup W$.
By P1, every face of $R$ is incident with at least one vertex of $U=W$ and, by P2,
every vertex of $U=W$ is incident with a face of $\FF$. So, each component of $H$
contains at least one face of $\FF$ and thus the number of components of $H$
is at most $|\FF|$.

As long as $H$ is not connected, proceed as follows. Choose one of the components.
By the definition of a cloud,
there exists a face $f$ in this component and a face $f'$ that is not contained in the chosen
component such that $f$ and $f'$ share a vertex. Add this vertex to $W$ and reset $H$ to be
the subgraph of $\VFodd(G)$ induced by $R\cup W$. Since in each step, the number of components
decreases by at least one, the final size of $W$ is at most $|U|+|\FF|-1\le 6|\FF|-2$.
Since $\FF$ is a collection of vertex-disjoint odd faces of $R$, $|\FF|\le\nu(R)$ and
the lemma follows.
\end{proof}

\section{Graphs with deadly faces}
\label{sec-tjoin}

To combine the results of this section with Lemma~\ref{lm-cloud}, we need to consider
signed plane graphs with distinguished faces which we refer to as {\em deadly}.
To be able to cope with them, we will first state a lemma
relating $T$-joins and $T$-cuts in bipartite graphs with deadly vertices,
which correspond to deadly faces in our application.

A {\em td-graph} is a plane connected bipartite graph $H$ with parts $A$ and $B$ and
two distinguished subsets $T$ and $D$ of $A$ such that $T$ has even size.
A $T$-cut of $H$ given by a vertex partition $(X,Y)$ is {\em nice}
if the end-vertex of every edge $e$ that is in $X$ is in $A$
the end vertex of $e$ in $Y$ is not adjacent to a vertex of $D$.
Two nice $T$-cuts are {\em vertex-disjoint}
if there is no vertex of $B$ incident with an edge of each of them.
If two nice $T$-cuts are vertex-disjoint, they must also be edge-disjoint.

\begin{lemma}
\label{lm-deadly}
Let $H$ be a td-graph with parts $A$ and $B$ and non-empty sets $T$ and $D$
with $T\subseteq D$.
If $|T|$ is even and $t$ is the minimum size of a $T$-join of $H$, then $H$ contains
at least $t/2-2|D|+2$ vertex-disjoint nice $T$-cuts.
\end{lemma}

\begin{proof}
We prove the following {\em claim} by induction on the size of a minimum $T$-join:
if $H$ is a td-graph with parts $A$ and $B$ and sets $T$ and $D$, $|T|$ is even and positive, and
$J$ is a minimum $T$-join of $H$, then $H$ contains at least
$|J|/2-2|D|+|T\cap D|-|T|+2$ vertex-disjoint nice $T$-cuts.
Since the statement of the lemma guarantees $T\subseteq D$,
this is enough to establish the lemma.

For the induction, consider a minimum $T$-join $J$ and the corresponding
moat solution $y$ of (Y).
The existence of $y$ is guaranteed by Proposition~\ref{prop-moats}.
During the induction step, we contract a connected subgraph to a new vertex in such a way
that the new vertex belongs to $T$ if and only if the contracted subgraph contains
an odd number of vertices of $T$. In this way, the $T$-cuts of the new graph correspond
to the $T$-cuts of the original graph.
We distinguish three (mutually excluding) cases:
\begin{itemize}
\item {\bf There exists a vertex $w\in T$ with $y_w\ge 2$.}
      Let $W$ be the set containing all vertices at distance at most two from $w$ (including the vertex $w$ itself).
      Contract the subgraph of $H$ induced by $W$ to a new vertex $w'$. Observe that the new graph $H'$
      is bipartite.
      Vertices of $H'$ different from $w'$ belong to $T'$ and $D'$ depending on their presence in the sets $T$ and $D$ in $H$.
      If $|W\cap T|$ is odd, $w'$ belongs to $T'$. If $(W\setminus\{w\})\cap D$ is non-empty,
      then $w'$ belongs to $D'$.
      So, $H'$ can be viewed as a td-graph with sets $T'$ and $D'$.

      Since $y$ is a moat solution, $J$ contains exactly two edges in the contracted subgraph.
      Let $J'$ be obtained from $J$ by contracting these two edges. It is easy to observe that
      $J'$ is a $T'$-join of the td-graph $H'$. Define $y'$ to be $y$ (with the provision that
      $y'_{S\setminus W\cup\{w'\}}=y_S$ for $W\subseteq S$) and, if $w'\in T'$, $y'_{w'}=y_w-2$.
      Note that if $w'\not\in T'$, then $W$ contains another vertex of $T$ and
      (Y) and the non-negativity of $y$ imply that $y_w\le 2$ which means $y_w=2$.
      So, $J'$ and the newly defined $y'$ are optimal solutions, i.e.,
      $J'$ is a minimum $T'$-join of $H'$.

      If $T'=\emptyset$ (so, we cannot apply induction), then $|J|=2$.
      If $D\not=\emptyset$, then we derive from $|D\cap T|\le|D|$ that
      $$|J|/2-2|D|+|T\cap D|-|T|+2\le |J|/2-1-|D|+|T\cap D|-|T|+2$$
      $$\le |J|/2-1-|T|+2=2-|T|\le 0\;\mbox{.}$$
      So, the claim holds. If $D=\emptyset$, then $|J|/2-2|D|+|T\cap D|-|T|+2=3-|T|\le 1$.
      Choose arbitrarily $w\in T$ and observe that the cut with parts $\{w\}$ and $V(H)\setminus\{w\}$
      is a nice $T$-cut. So, the claim also holds.

      If $T'\not=\emptyset$, we invoke the induction.
      So, $H'$ contains a collection $\CC'$ of $|J'|/2-2|D'|+|T'\cap D'|-|T'|+2$ vertex-disjoint nice $T'$-cuts.
      This collection corresponds to a collection $\CC$ of nice $T$-cuts in $H$ (recall that if any vertex
      of $W\setminus\{w\}$ belongs to $D$, then $w'$ is in $D'$). In case that $W$ contains no vertex of $D$,
      enhance the collection $\CC$ by the (nice) $T$-cut with parts $\{w\}$ and $V(H)\setminus\{w\}$.
      Observe that the cuts of $\CC$ are still vertex-disjoint.

      It remains to argue that $\CC$ contains at least $|J|/2-2|D|+|T\cap D|-|T|+2$ cuts.
      Observe that $|J'|=|J|-2$ and $|D'|\le |D|$.
      We distinguish several cases:
      \begin{itemize}
      \item {\bf The sets $W$ and $D$ are disjoint.}	    
            It holds that $|D'|=|D|$ and $|T'\cap D'|=|T\cap D|$. Using the inequality $|T'|\le |T|$, we obtain that
	    \begin{eqnarray*}
            |\CC|=|\CC'|+1&=&|J'|/2-2|D'|+|T'\cap D'|-|T'|+3\\
            &\ge& |J|/2-2|D|+|T\cap D|-|T|+2\;\mbox{.}
            \end{eqnarray*}
      \item {\bf The only vertex of $W$ in $D$ is $w$.}
            It holds that $|D'|=|D|-1$ and $|T'\cap D'|=|T\cap D|-1$. Since $|T'|\le |T|$, we get that
	    \begin{eqnarray*}
            |\CC|=|\CC'|&=&|J'|/2-2|D'|+|T'\cap D'|-|T'|+2\\
            &\ge &|J|/2-1-2(|D|-1)+(|T\cap D|-1)-|T|+2\\
	    &=&|J|/2-2|D|+|T\cap D|-|T|+2\;\mbox{.}
            \end{eqnarray*}
      \item {\bf $D$ contains a vertex of $W\setminus\{w\}$ and in addition $w'\in T'$ and $w\in D$.}
            Since $W$ contains two vertices of $D$,
            the size of $D'$ is strictly smaller than the size of $D$.
	    Now observe that
            \begin{equation}
	    |T\cap D|-|T'\cap D'|=|(W\setminus\{w\})\cap T\cap D|\le |T|-|T'|\;\mbox{.}\label{eq-1}
	    \end{equation}
            So, we obtain from (\ref{eq-1}) that
	    \begin{eqnarray*}
            |\CC|=|\CC'|&=&|J'|/2-2|D'|+|T'\cap D'|-|T'|+2\\
            &\ge& |J|/2-1-2(|D|-1)+|T\cap D|-|T|+2\\
	    &>&|J|/2-2|D|+|T\cap D|-|T|+2\;\mbox{.}
            \end{eqnarray*}
      \item {\bf $D$ contains a vertex of $W\setminus\{w\}$ and in addition $w'\in T'$ and $w\not\in D$.}
	    Since $w$ is in $T$ but not in $D$ and $w'$ is both in $T'$ and $D'$, we obtain that
            \begin{equation}
	    |T\cap D|-(|T'\cap D'|-1)=|(W\setminus\{w\})\cap T\cap D|\le |T|-|T'|\;\mbox{.}\label{eq-2}
	    \end{equation}
	    The equation (\ref{eq-2}) implies that
	    \begin{eqnarray*}
            |\CC|=|\CC'|&=&|J'|/2-2|D'|+|T'\cap D'|-|T'|+2\\
            &\ge& |J|/2-1-2|D|+|T\cap D|-|T|+3\\
	    &=&|J|/2-2|D|+|T\cap D|-|T|+2\;\mbox{.}
            \end{eqnarray*}
      \item {\bf $D$ contains a vertex of $W\setminus\{w\}$ and $w'\not\in T'$.}
            If $|W\cap D|\ge 2$, then $|D'|\le |D|-1$. Now observe that
	    \begin{equation}
	    |T\cap D|-|T'\cap D'|=|W\cap T\cap D|\le |T|-|T'|\;\mbox{.}\label{eq-3}
	    \end{equation}
	    We combine (\ref{eq-3}) and $|D'|\le |D|-1$ to obtain the desired estimate
	    \begin{eqnarray*}
            |\CC|=|\CC'|&=&|J'|/2-2|D'|+|T'\cap D'|-|T'|+2\\
            &\ge& |J|/2-1-2(|D|-1)+|T\cap D|-|T|+2\\
	    &>&|J|/2-2|D|+|T\cap D|-|T|+2\;\mbox{.}
            \end{eqnarray*}
	    If $|W\cap D|=1$, then $w\not\in D$ by the case assumption. So, we can strengthen (\ref{eq-3}):
	    \begin{equation}
	    |T\cap D|-|T'\cap D'|=|W\cap T\cap D|\le |T|-|T'|-1\;\mbox{.}\label{eq-4}
	    \end{equation}
            We now combine (\ref{eq-4}) and $|D'|\le |D|$ to derive
	    \begin{eqnarray*}
	    |\CC|=|\CC'|&=&|J'|/2-2|D'|+|T'\cap D'|-|T'|+2\\
	    &\ge&|J|/2-1-2|D|+|T\cap D|-|T|+3\\
	    &\ge&|J|/2-2|D|+|T\cap D|-|T|+2\;\mbox{.}
	    \end{eqnarray*}
      \end{itemize}
\item {\bf There exists $S\in\OO(T)$ with $y_S\ge 1$ and $y_w\le 1$ for all $w\in T$.}
      Consider an inclusion-wise minimal subset $S\subseteq T$ with $y_S\ge 1$.
      Since $y$ is a moat solution and $H$ is bipartite,
      it holds that $y_w=1$ for every $w\in S$.
      Let $H'$ be the graph obtained from $H$ by identifying the vertices of $S$ and
      let $w'$ be the new vertex. Define $T'$ to be the set containing $w'$ and all vertices of $H'$ that are in $T$.
      Define $D'$ to be the set containing all vertices of $H'$ in $D$; add $w'$ to $D'$ if $S$ contains at least one vertex from $D$.

      Let $J'$ be the $T'$-join of $H'$ obtained from $J$ by removing the $(|S|-1)/2$ new pairs of parallel edges of $J$
      arising by identification of the vertices of $S$. The size of $J'$ is thus $|J|-|S|+1$.
      Define $y'_{w'}=y_S+1$ and $y'$ equal to $y$ otherwise (setting
      $y'_{S'\setminus S\cup\{w'\}}=y_{S'}$ for $S'\supset S$). So, $J'$ and $y'$ are optimal,
      i.e., $J'$ is a minimum $T'$-join of $H'$.

      By induction, we obtain that $H'$ contains a collection $\CC$ of at least $|J'|/2-2|D'|+|T'\cap D'|-|T'|+2$
      vertex-disjoint nice $T'$-cuts. These cuts also form a collection of vertex-disjoint nice $T$-cuts in $H$.
      
      We claim that $\CC$ consists of at least $|J|/2-2|D|+|T\cap D|-|T|+2$ cuts.
      First observe that $|T\cap D|-|T'\cap D'|=|D|-|D'|$ and $|T'|=|T|-(|S|-1)$.
      Hence, the number of cuts in $\CC$ is at least
      \begin{eqnarray*}
      &&|J'|/2-2|D'|+|T'\cap D'|-|T'|+2\\
      &&\ge|J'|/2-|D|-|D'|+|T'\cap D'|-|T'|+2\\
                                   &&\ge|J|/2-(|S|-1)/2-2|D|+|T\cap D|-|T|+(|S|-1)+2\\
				   &&>|J|/2-2|D|+|T\cap D|-|T|+2\;\mbox{.}
      \end{eqnarray*}				   
\item {\bf It holds that $y_w\le 1$ for every $w\in T$ and $y_S=0$ for every $S\in\OO(T)$.}
      Observe that $|J|=\sum_{w\in T}y_w\le |T|$. On the other hand, $|J|$ must contain at least $|T|$ edges
      since the distance between any two vertices of $T$ is at least two. So, $|J|=|T|$.
      The equality $|J|=|T|$ and the inequality $|T\cap D|\le |D|$ now imply
      that $|J|/2-2|D|+|T\cap D|-|T|+2\le 2-|D|-|T|/2$. So, if $D$ is non-empty,
      the claim holds. If $D$ is empty, then $2-|D|-|T|/2\le 1$.
      Choose now a vertex $w$ in $T$ arbitrarily and consider the $T$-cut with parts $\{w\}$ and $V(H)\setminus\{w\}$.
      Since this is a nice $T$-cut, the claim follows.
\end{itemize}
The proof of the lemma is now finished.
\end{proof}

We continue with the following lemma implicit in~\cite{bib-fiorini07+}. We include the proof
for completeness.

\begin{lemma}
\label{lm-tjoin}
Let $G$ be a signed plane graph and let $T$ be the set of its odd faces.
It holds that $\tau(G)\le t/2$ where $t$ is the minimum size of a $T$-join
in $\VF(G)$.
\end{lemma}

\begin{proof}
Fix a minimum $T$-join in $\VF(G)$ and let $W$ be the set of vertices of $G$ (which are
also vertices of $\VF(G)$) incident with an edge of the $T$-join.
Since no vertex of $W$ is in $T$, every vertex of $W$ is incident with at least
two edges of the $T$-join.
Since $\VF(G)$ is bipartite,
each edge of the $T$-join is incident with at most one vertex of $W$ and
we get that $|W|\le t/2$. We show that every odd cycle of $G$ contains a vertex of $W$.

Let $C$ be an odd cycle of $G$, i.e., the interior of $C$ contains an odd number of odd faces.
Since the vertices of $C$ form a vertex cut in $\VF(G)$ and the number of odd faces inside it
is odd, at least one of the vertices of $C$ in $\VF(G)$ is incident with edges of the $T$-join.
Such a vertex is contained in $W$.
\end{proof}

Lemmas~\ref{lm-deadly} and~\ref{lm-tjoin} yield the following theorem.

\begin{theorem}
\label{thm-deadly}
Let $G$ be a signed plane graph with some faces marked as deadly.
If $G$ contains an odd face and every odd face is deadly, then $\tau(G)\le\nu_{\dead}(G)+2d-2$
where $\nu_{\dead}(G)$ is the maximum number of vertex-disjoint odd cycles that
are vertex-disjoint from deadly faces and $d$ is the number of deadly faces.
\end{theorem}

\begin{proof}
Let $H$ be the vertex-face incidence graph $\VF(G)$ of $G$ and
let $A$ be its part corresponding to faces.
Set $D\subseteq A$ to be the vertices of $H$ corresponding to deadly faces and $T\subseteq A$
those corresponding to odd faces.
So, $H$ is a td-graph.

By the assumption of the theorem, $T$ is a subset of $D$.
Let $J$ be a minimum $T$-join of $H$. By Lemma~\ref{lm-deadly},
$H$ contains $|J|/2-2|D|+2=|J|/2-2d+2$ vertex-disjoint nice $T$-cuts. Let $\CC$ be a collection
of such cuts.

Consider a nice $T$-cut given by a vertex partition $(X,Y)$ and let $\FF$ be the set of faces corresponding
to vertices in $X$. Consider the symmetric difference of the boundary cycles of the faces of $\FF$.
Since the considered cut is a $T$-cut, $\FF$ contains an odd number of odd faces.
Consequently, the symmetric difference is formed by a union of cycles such that at least
one of the cycles is odd. So, every cut of $\CC$ gives rise to an odd cycle. Since the cuts of $\CC$
are vertex-disjoint, these odd cycles are vertex-disjoint, and since the cuts are nice,
the cycles do not share a vertex with a deadly face. We conclude that $G$ contains at least
$|J|/2-2d+2$ vertex-disjoint odd cycles disjoint from deadly faces, i.e., $\nu_{\dead}(G)\ge |J|/2-2d+2$.

On the other hand, $\tau(G)\le |J|/2$ by Lemma~\ref{lm-tjoin} which finishes the proof.
\end{proof}

\section{Main result}

We first combine Lemma~\ref{lm-cloud} and Theorem~\ref{thm-deadly} to prove our bound for $3$-connected plane signed graphs.

\begin{theorem}
\label{thm-3conn}
Let $G$ be a $3$-connected plane signed graph. If $G$ contains an odd cycle,
then $\tau(G)\le 6\nu(G)-2$.
\end{theorem}

\begin{proof}
Let $R_1,\ldots,R_k$ be the clouds of $G$. By Lemma~\ref{lm-cloud},
for each $i=1,\ldots,k$, there exists a set $W_i$ of vertices and a collection $\CC_i$
of vertex-disjoint odd faces of $R_i$ such that $|W_i|\le 6|\CC_i|-2$ and
$R_i\cup W_i$ induces a connected subgraph of $\VFodd(G)$.

Let $G'$ be the graph obtained from $G$ by removing the vertices of $W_1\cup\cdots\cup W_k$.
If $G'$ has no odd faces (which implies it has no odd cycles),
then $W_1\cup\cdots\cup W_k$ is a transversal of $G$.
Since $\CC_1\cup\cdots\cup\CC_k$ is a packing of odd cycles in $G$, we get the following:
$$\tau(G)\le\left|\bigcup_{i=1}^k W_i\right|=\sum_{i=1}^k|W_k|$$
$$\le\sum_{i=1}^k (6|\CC_i|-2)\le 6\left|\bigcup_{i=1}^k \CC_i\right|-2\le 6\nu(G)-2\;\mbox{.}$$
So, we can assume that $G'$ has an odd face.

Mark the new faces of $G'$, i.e., those containing a region bounded by an odd face of $G$, as deadly.
Since each $R_i\cup W_i$ induces a connected subgraph of $\VFodd(G)$,
the regions bounded by the faces of the same cloud $R_i$ are now contained
in the same face of $G'$. Hence, the number of deadly faces does not exceed $k$.
We now apply Theorem~\ref{thm-deadly} to $G'$. So, there exists
a set $W_0$ of vertices of $G'$ such that $G'\setminus W_0$ has no odd cycle and
a collection $\CC_0$ of vertex-disjoint odd cycles of $G'$ disjoint from deadly faces such that
$|W_0|\le|\CC_0|+2k-2$. In particular, the union $W_0\cup W_1\cup\cdots\cup W_k$ is a transversal of $G$ and
the union $\CC_0\cup\CC_1\cup\cdots\cup\CC_k$ is a packing of odd cycles.
We now relate the sizes of the two unions:
$$\tau(G)\le\left|\bigcup_{i=0}^k W_i\right|=|W_0|+\sum_{i=1}^k|W_k|\le |\CC_0|+2k-2+\sum_{i=1}^k (6|\CC_i|-2)$$
$$\le\left(\sum_{i=0}^k 6|\CC_i|\right)-2=6\left|\bigcup_{i=0}^k \CC_i\right|-2\le 6\nu(G)-2\;\mbox{.}$$
This completes the proof of the theorem.
\end{proof}

We now prove our main result.

\begin{theorem}
\label{thm-main}
Let $G$ be a plane signed graph. It holds that $\tau(G)\le 6\nu(G)$.
\end{theorem}

\begin{proof}
If $G$ has no odd cycles, there is nothing to prove.
So, we assume that $G$ has an odd cycle.
We prove by the induction on the number of vertices of $G$ the following: if $G$ contains an odd cycle,
then $\tau(G)\le 6\nu(G)-2$.
If $G$ is $3$-connected, the estimate follows from Theorem~\ref{thm-3conn}.
If $G$ is not connected, then apply induction to each connected component containing an odd cycle to get the desired inequality.
So, we assume that $G$ is connected but it contains a vertex cut of size one or two.

Let $C$ be a minimum vertex cut of $G$ and
let $G_1$ and $G_2$ be two non-trivial subgraphs of $G$ such that $G_1$ and $G_2$ intersect at $C$ only and
their union is $G$. If $|C|=2$ and the two vertices of $C$ are adjacent, we include the edge between them to both $G_1$ and $G_2$.

If neither $G_1\setminus C$ nor $G_2\setminus C$ contains an odd cycle, then $\tau(G)\le 2$ since $C$ is a transversal.
Consequently, $\tau(G)\le 6\nu(G)-2$ because $\nu(G)\ge 1$.

If both $G_1\setminus C$ and $G_2\setminus C$ have an odd cycle, then apply induction to each of them.
This yields transversals $T_1$ and $T_2$ and packings $\CC_1$ and $\CC_2$ of odd cycles in $G_1\setminus C$ and $G_2\setminus C$, respectively,
such that $|T_i|\le 6|\CC_i|-2$ for $i=1,2$.
Since $T_1\cup T_2\cup C$ is a transversal of $G$ and $\CC_1\cup\CC_2$ is a packing, we get the following:
$$\tau(G)\le |T_1|+|T_2|+|C|\le 6|\CC_1|-2+6|\CC_2|-2+2\le 6|\CC_1\cup\CC_2|-2\le 6\nu(G)-2\;\mbox{.}$$

It remains to consider (by symmetry) the case that $G_1\setminus C$ has an odd cycle and $G_2\setminus C$ has no odd cycle.
If $G_2$ has an odd cycle, apply induction to $G_1\setminus C$ to get a transversal $T_1$ and a packing $\CC_1$ of odd cycles of $G_1\setminus C$ such that $|T_1|\le 6|\CC_1|-2$.
Since $T_1\cup C$ is a transversal of $G$ and the packing $\CC_1$ can be extended by an odd cycle of $G_2$ to a packing of $G$,
we obtain that
$$\tau(G)\le |T_1|+|C|\le 6|\CC_1|-2+2=6(|\CC_1|+1)-6\le 6\nu(G)-6\le 6\nu(G)-2\;\mbox{.}$$
Hence, we assume that $G_2$ has no odd cycle. If $|C|=2$, then all paths between the two vertices of $C$ in $G_2$ have the same parity.
So, if $|C|=2$ and the two vertices of $C$ are not adjacent, let $G'_1$ be the graph obtained from $G_1$ by adding an edge between the two vertices of $C$
with the parity equal to the common parity of the paths between them in $G_2$. Otherwise, let $G'_1$ be $G_1$. By induction, $G'_1$ contains
a transversal $T_1$ and a packing $\CC_1$ of odd cycles such that $|T_1|\le 6|\CC_1|-2$. The packing $\CC_1$ gives rise to a packing of the same
size in $G$ since we can reroute a possible cycle using the added edge through the interior of $G_2$. We now argue that $T_1$ is a transversal of $G$.
Consider an odd cycle of $G\setminus T_1$. This cycle must include a vertex of $G_1\setminus C$ (there are no odd cycles avoiding vertices of $G_1\setminus C$ in $G$) and
a vertex of $G_2\setminus C$ (otherwise, $T_1$ is not a transversal of $G_1\subseteq G'_1$). In particular, $|C|=2$ and the cycle includes both vertices of $C$.
Rerouting the cycle through the edge between the vertices of $C$ yields an odd cycle in $G'_1\setminus T_1$ contrary to the fact that $T_1$ is a transversal of $G'_1$.
We conclude that $T_1$ is a transversal of $G$ and since $G$ has a packing of $|\CC_1|$ odd cycles, the theorem follows.
\end{proof}

A close inspection of the proofs presented in this paper yields that all their steps can be efficiently performed,
i.e., there exists a polynomial-time algorithm that for a given (signed) planar graph $G$ returns a collection $\CC$ of vertex-disjoint odd cycles and
a set of vertices $W$ such that $G\setminus W$ has no odd cycle and $|W|\le 6|\CC|$. So, the next corollary follows.

\begin{corollary}
\label{cor-alg}
There exists a polynomial time algorithm that for a given planar (signed) graph $G$ returns a collection of vertex-disjoint
odd cycles of size at least $\nu(G)/6$.
\end{corollary}

\end{document}